\documentclass[11pt,reqno]{amsart}
\pdfoutput=1
\usepackage{geometry}   
\usepackage{graphicx}
\usepackage{amssymb}
\usepackage{epstopdf} 
\usepackage{hyperref}
\usepackage{graphicx, amsmath, amsthm, latexsym, amssymb, amsfonts, epsfig, url, paralist}
\usepackage{xcolor}

\newtheorem{theorem}{Theorem}[section]
\newtheorem{lemma}[theorem]{Lemma}
\newtheorem{proposition}[theorem]{Proposition}
\newtheorem{corollary}[theorem]{Corollary}
\newtheorem{conjecture}[theorem]{Conjecture}
\theoremstyle{definition}
\newtheorem{definition}[theorem]{Definition}

\newtheorem{example}[theorem]{Example}

\newcommand\C{\mathbb C}
\newcommand\R{\mathbb R}

\newcommand\Z{\mathbb Z}

\newcommand{\conv}{\operatorname{conv}}
\newcommand{\graph}{\operatorname{graph}}
\newcommand{\skel}{\operatorname{skel}}

\title[Systems of permutations not realizable by triangulations of $\Delta_{n-1}\times \Delta_{d-1}$]
{Some acyclic systems of permutations are not realizable by triangulations of a product of simplices}

\author{Francisco Santos}
\address[Francisco Santos]{Facultad de Ciencias, Universidad de Cantabria, 
Av. de los Castros s/n, E-39005 Santander, Spain.}  
\email{francisco.santos@unican.es}

\thanks{Partially supported by the Spanish Ministry of Science through grants MTM2011-22792 and CSD2006-00032 (i-MATH). This paper answers half of the question posed by F.~Ardila in the open problems session of the \emph{Workshop on Tropical Geometry} held at the CIEM (Castro Urdiales, Spain) in December 2011. I thank the organizers for assembling such an interesting group of people and talks}

\makeindex

\begin{document}

\begin{abstract}
The \emph{acyclic system conjecture} of Ardila and Ceballos can be interpreted as saying the following:
``Every triangulation of the $3$-skeleton of a product $\Delta_k\times \Delta_l$
of two simplices can be extended to a triangulation of the whole product".
We show an example disproving this.

Motivation for this conjecture comes from a related conjecture, the ``spread-out simplices'' conjecture of Ardila and Billey. We give some necessary conditions that counter-examples to this second conjecture (if they exist) must satisfy.
\end{abstract}

\maketitle


\section{Introduction}
\label{sec:intro}

Triangulations (with no extra vertices) of  the product of two simplices have extensive applications and implications in geometric and algebraic combinatorics, optimization, tropical geometry, and in other areas. See, for example, the references in~\cite{ArCe11, ArBi07, Santos:cayley}, and~\cite[Section 6.2]{dLRaSa10}. 

Since triangulations of $\Delta_{n-1}\times \Delta_1$ are in bijection with permutations of $[n]$, every triangulation $T$ of $\Delta_{n-1}\times \Delta_{d-1}$ induces a  \emph{system of permutations} on $K_d=\graph(\Delta_{d-1})$, as follows: Along each edge $e$ of $\graph(\Delta_{d-1})$ we write the permutation of $[n]$ that corresponds to the restriction of $T$ to $\Delta_{n-1}\times e$. We say ``write along'' because the edge $e$ is considered oriented, and reversing the orientation of $e$ amounts to reversing the permutation ``written on it''. (A permutation in this paper is merely an ordering of the symbols $1$ to $n$; we are not concerned with the group structure or other algebraic properties of permutations. Reversing means reordering the symbols in the opposite way). See details in Section~\ref{sec:permutations}. 

Ardila and Ceballos~\cite{ArCe11} try to answer the following question: if we are given a system of permutations of $[n]$ on the complete graph $K_d=\graph(\Delta_{d-1})$, 
what are the conditions for it to actually come from a triangulation of $\Delta_{n-1}\times \Delta_{d-1}$?
A necessary condition that they identify is that the system of permutations needs to be \emph{acyclic}: for every $i,j\in [n]$, if we reorient
 $\graph(\Delta_{d-1})$ so that \emph{$i$ comes before $j$} along every edge, the directed graph is acyclic. The work of Ardila and Ceballos implies the following: 

\begin{lemma}[Ardila and Ceballos~\cite{ArCe11}]
\label{lemma:ArCe11}
Let $\Sigma$ be a system of permutations of $[n]$ on $\graph(\Delta_{d-1})$ or, equivalently, a triangulation of the polyhedral complex $\Delta_{n-1}\times\graph(\Delta_{d-1})$. The following properties are equivalent:
\begin{enumerate}
\item The system of permutations is acyclic.
\item It has a dual system of permutations of $[d]$ on $\graph(\Delta_{n-1})$. That is, the triangulation of $\Delta_{n-1}\times\graph(\Delta_{d-1})$ is compatible with a triangulation of $\graph(\Delta_{n-1})\times \Delta_{d-1}$ (and then this dual is unique).
\item The triangulation extends to a triangulation of $\Delta_{n-1}\times\skel_2(\Delta_{d-1})$.
\item The triangulation is compatible with a triangulation of $\skel_3(\Delta_{n-1}\times\Delta_{d-1})$
\end{enumerate}
\end{lemma}

Ardila and Ceballos conjectured that in fact every acyclic system of permutations extends to a triangulation of $\Delta_{n-1}\times\Delta_{d-1}$. The main result of this paper is that this is false starting (at least) in $\Delta_4\times \Delta_3$ (Section~\ref{sec:acyclic}). Before that, and as a partial result, we show that there is a triangulation of the $4$-skeleton of $\Delta_2\times \Delta_3$ that does not extend to a triangulation of 
$\Delta_2\times \Delta_3$ (Section~\ref{sec:boundary}).

Incidentally, our examples, and the equivalence of parts (3) and (4) in Lemma~\ref{lemma:ArCe11}, imply the following. 
\begin{quote}
For $i\in\{1,3\}$, every triangulation of $\skel_i(\Delta_{n-1}\times\Delta_{d-1})$ extends to a triangulation of
$\skel_{i+1}(\Delta_{n-1}\times\Delta_{d-1})$. For $i\in \{2,4,6\}$, the same is not true. (For $i=2$ consider the example of a triangular prism with its three squares triangulated in a cyclic way. For $i=4$ and $i=6$ consider our examples from Sections~\ref{sec:boundary} and ~\ref{sec:acyclic}).
\end{quote} 
We wonder whether the different behavior depending on the parity of $i$ is just a coincidence, or it continues for bigger values of $i$. In particular, it would be interesting to solve the case $i=5$.

\medskip

But the main motivation for Ardila and Ceballos to study acyclic systems of permutations was to use them as an intermediate tool to try to prove the \emph{spread-out simplices} conjecture of Ardila and Billey~\cite{ArBi07}, connected with the matroid of lines in an arrangement of complete flags in $\C^n$ (see Section~\ref{sec:spread}). In Section~\ref{sec:simplices} we show that our counter-example is \emph{not} a counter-example to the spread-out simplices conjecture. While doing this, we identify certain sufficient conditions for a spread-out system of positions to be \emph{realizable} by some  fine mixed subdivision.

Section~\ref{sec:prelim} contains known facts on triangulations of $\Delta_{n-1}\times \Delta_{d-1}$ and fine mixed subdivisions of $n\Delta_{d-1}$. In particular, it reviews some of the results from~\cite{ArCe11, ArBi07}.

\section{Preliminaries}
\label{sec:prelim}
\subsection{Triangulations of $\Delta_{n-1}\times\Delta_{d-1}$ and fine mixed subdivisions of $n\Delta_{d-1}$}
\label{sec:calyley}
%
%
Let $T$ be a triangulation of  $\Delta_{n-1}\times\Delta_{d-1}$. To each cell $B\in T$ we associate the $n$-tuple $(B_1,\dots,B_n)$ of faces of $\Delta_{d-1}$ that it uses on the different vertex-fibers of the projection $\Delta_{n-1}\times\Delta_{d-1}\to \Delta_{n-1}$. Put differently, 
if $v_1,\dots,v_n$ denote the vertices of $\Delta_{n-1}$, $B_i$ is the face of $\Delta_{d-1}$ for which
\[
\{v_i\}\times B_i = B\cap (\{v_i\}\times \Delta_{d-1}). 
\]

The cells $\{B_1+\cdots+B_n: B \in T\}$ form a \emph{mixed subdivision} $\mathcal T$ of $n\Delta_{d-1}$: a polyhedral decomposition of $n\Delta_{d-1}$ into cells each of which is a \emph{Minkowski sum} of $n$ faces of $\Delta_{d-1}$. The mixed subdivision corresponding to a triangulation is \emph{fine}, meaning that in each Minkowski cell $\sum B_i$ we have that $\dim(\sum B_i) = \sum (\dim B_i)$. The following statement is a special case of the \emph{Cayley Trick}~\cite{HuRaSa:cayley}:

\begin{theorem}[\cite{HuRaSa:cayley, Santos:cayley}]
The above correspondence produces a bijection between triangulations of $\Delta_{n-1}\times\Delta_{d-1}$ and fine mixed subdivisions of $n\Delta_{d-1}$.
\end{theorem}

In a mixed subdivision, cells come with a natural \emph{ordered Minkowski sum} structure; that is, strictly speaking the cells 
of a mixed subdivision are the  $n$-tuples $(B_1,\dots,B_n)$ rather than their Minkowski sums, even if we normally write them as Minkowski sums $\sum B_i$ to simplify notation. Cells in a mixed subdivision intersect face to face in the following labeled sense: if $\sum B_i$ and $\sum C_i$ are two such cells then we have that $B_i\cap C_i$ is a face of both $B_i$ and $C_i$ for every $i=1,\dots, n$, and $\sum (B_i\cap C_i)$ is also a cell in the mixed subdivision.

In every fine mixed subdivision of $n\Delta_{d-1}$ there are some special cells which appear as $(d-1)$-simplices because they are
the Minkowski sum of $\Delta_{d-1}$ with $n-1$ vertices. We call them the \emph{unmixed simplices} of the subdivision.
There are $n$ of them, one with the simplex summand in each possible position. It turns out~\cite[Theorem 2.6]{Santos:cayley} that labeling these cells with the symbols $1$ to $n$ is enough to recover from an ``unlabeled'' fine mixed subdivision (a mere decomposition of $n\Delta_{d-1}$ into subpolytopes with individual Minkowski decompositions) the whole labeled one (the assignment of an $n$-tuple $(B_1,\dots,B_n)$ of faces of $\Delta_{d-1}$ to each cell so that they intersect face to face in the labeled sense).

\begin{example}[A mixed subdivision of $4\Delta_2$]
\label{ex:4D2}
A fine mixed subdivision of $n\Delta_d$ is a \emph{lozenge tiling}: 
a decomposition of $n\Delta_2$ into cells which are either translated copies of $\Delta_2$ or rhombi (also known as \emph{lozenges}) which are the union of a translated copy of $\Delta_2$ and a translated copy of $-\Delta_2$. Each such tiling has $n$ cells which are triangles, because the triangular tiling of $n\Delta_2$ has $n$ more copies of $\Delta_2$ than of $-\Delta_2$. Each of them is the center of a \emph{zone}, built by starting with the triangle itself and
recursively adding to it lozenges in the three directions, until we reach the three sides of $n\Delta_2$.
The left part of Figure~\ref{fig:example} shows a mixed subdivision of $4\Delta_2$ in the ``unlabeled sense''; on the right we have labeled the triangles with the numbers from $1$ to $4$ and the zones induced are shown in Figure~\ref{fig:zones}.

\begin{figure}
	\centering
	\includegraphics[scale=0.7]
	{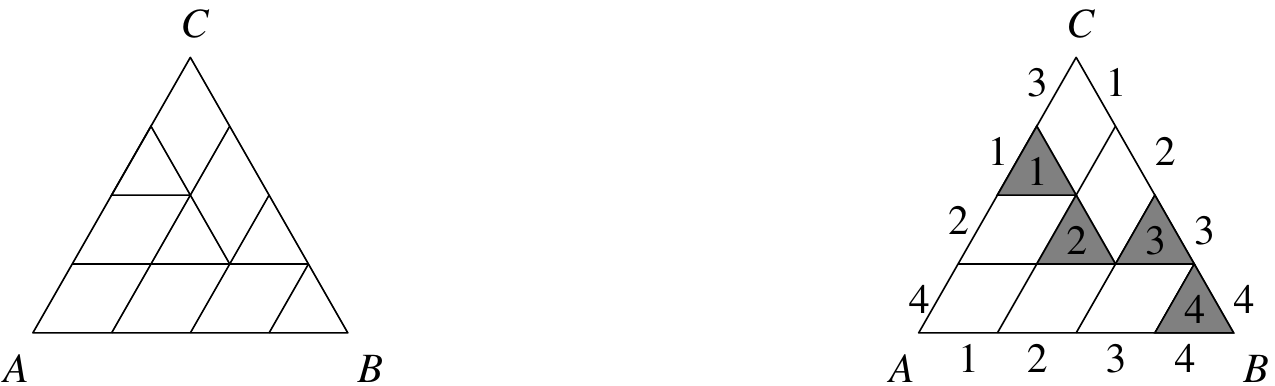}
	\bigskip\\	
	\caption{An ``unlabeled'' fine mixed subdivision of $4\Delta_2$ (left). Labeling the four triangular cells with the symbols $1$ to $4$ specifies uniquely a (labeled) fine mixed subdivision. In particular, it gives the corresponding system of permutations (right)}
	\label{fig:example}
\end{figure}

\begin{figure}[ht]
	\centering
	\includegraphics[scale=0.7]
	{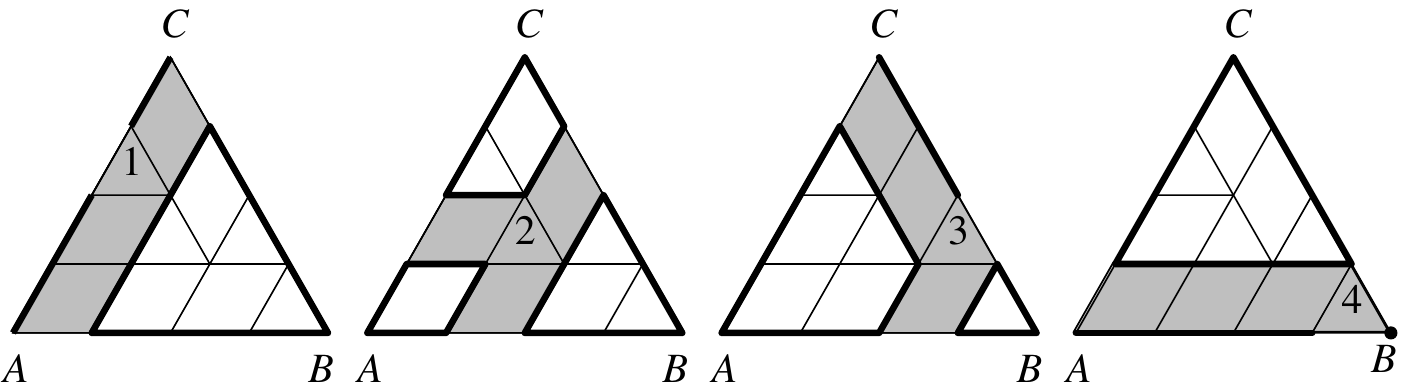}
	\bigskip\\	
	\caption{The $4$ zones in a fine mixed subdivision of $4\Delta_2$.}
	\label{fig:zones}
\end{figure}

From the zones of a mixed subdivision of $n\Delta_2$ we can recover the following information:
\begin{itemize}
\item \emph{The acyclic system of permutations}. Along each edge of $n\Delta_2$ we see $n$ segments, each of which belongs to a different zone. The permutation associated to that edge is precisely the sequence of zones. For example, the permutations on the edges $AB$, $AC$ and $BC$ for the subdivision of Figures~\ref{fig:example} and~\ref{fig:zones} are $1234$, $4213$ and $4321$ (here and elsewhere, we denote by the first capital letters, $A$, $B$ and $C$ in this case, the vertices of $\Delta_d-1$).

\item \emph{The Minkowski labeling of each cell}. The $i$-th (open) zone consists of the cells whose $i$-th Minkowski summand is more than a single vertex: the Minkowski summand is the whole triangle in the triangle of the zone, and it is the edge $AB$, $BC$ or $AC$ in the rhombi of the three arms, respectively; if a cell is not in the $i$-th zone then it is in one of the three complementary (closed) regions, each of which contains one of the three vertex of $n\Delta_2$. In this case, the $i$-th summand of that cell is the corresponding vertex of $\Delta_2$. 

For example, the upper most rhombus in Figures~\ref{fig:example} and~\ref{fig:zones} has the Minkowski decomposition $BC + C + AB + C$.
\end{itemize}
\end{example}

What we say in the $d=2$ case holds (with appropriate definitions) in every dimension. Each unmixed simplex is the center of a \emph{zone} that can be built by extending from the simplex in all directions. The $i$-th zone contains all the cells whose $i$-th summand is more than a single vertex, and from the zones it is very easy to recover the Minkowski sum labeling of every cell.

\subsection{The acyclic system of permutations of a triangulation of $\Delta_{n-1}\times\Delta_{d-1}$}
\label{sec:permutations}

To better understand acyclic systems of permutations let us analyze the simplest non-trivial case, that of the triangular prism $\Delta_2\times\Delta_1$. It is well-known that it has six triangulations, corresponding to the six permutations of the vertices of $\Delta_2$. Each of them is characterized by the diagonals it introduces in the three boundary squares. More precisely, out of the $2^3$ possible choices of one diagonal in each square, the six ``non-cyclic'' ones extend to triangulations of the prism and the two cyclic ones do not. If we denote  $\{1,2,3\}$ the vertices of  $\Delta_2$ and $\{A,B\}$ those of $\Delta_1$, the information on what diagonals we choose can be encoded as an orientation of the complete graph $K_3$ on $\{1,2,3\}$ with the following meaning:
\begin{quote}
The edge $ij$ is oriented from $i$ to $j$ if the quadrilateral $AB \times ij$ is triangulated with the diagonal $(A,i)(B,j)$.
\end{quote}
In this way, the six valid choices of diagonals correspond to the acyclic orientations of $K_3$.

Let us now move to the slightly more general case of a prism over a simplex. That is, let $P=\Delta_{n-1}\times \Delta_1$, with vertices labeled $[n]:=\{1,\dots,n\}$ and $\{A,B\}$, respectively. Let $T$ be a triangulation of $P$. For each edge $ij$ of $\Delta_{n-1}$ we encode as before which diagonal of the square $AB\times ij$ is used in $T$. In this way, $T$ induces an orientation of the $1$-skeleton of $\Delta_{n-1}$ (the complete graph on $[n]$). By what we said before, this orientation does not contain cycles of length $3$ (which would correspond to triangular prisms with their boundary triangulated in an incompatible way). Now, an orientation of th complete graph without $3$-cycles must necessarily be acyclic. So, we can regard it as an ordering (permutation) of the $n$ vertices of $\Delta_{n-1}$. It is well-known that:

\begin{lemma}[\protect{\cite[Proposition 6.2.3]{dLRaSa10}}]
$\Delta_{n-1}\times \Delta_1$ has exactly $n!$ different triangulations. They are in bijection, via the above rule, to the $n!$ permutations of the vertices of $\Delta_{n-1}$.
\end{lemma}

Finally, let us consider the general case of $P=\Delta_{n-1}\times \Delta_{d-1}$. Let $[n]$ be the set of vertices of $\Delta_{n-1}$ and let $S$ denote that of $\Delta_{d-1}$. Let $T$ be a triangulation of $P$. As before, for each edge $IJ$ of $\Delta_{d-1}$, $T$ induces a triangulation of $\Delta_{n-1}\times IJ$, which we encode as a permutation of $[n]$. It is important to notice that the edge $IJ$ is considered oriented and that changing its orientation reverses the permutation. Graphically, we consider the permutation of $[n]$ as written ``along the edge'' $IJ$, so that we can read it from $I$ to $J$ or from $J$ to $I$.

\begin{definition}
A \emph{system of permutations} of $[n]$ on a graph $G$ consists of one permutation of $[n]$ ``written along'' each edge of $G$. 
A  system of permutations of $[n]$ is \emph{acyclic} if, for every two symbols $i,j\in [n]$, the orientation of $G$ obtained directing every edge from $i$ to $j$ is acyclic.
\end{definition}

We are only interested in the case where $G$ is the complete graph (the $1$-skeleton of a simplex). In this case the system is acyclic if and only if it is acyclic when restricted to every triangle of the graph.

Of course, we can exchange the roles of $\Delta_{n-1}$ and  $\Delta_{d-1}$. So, every triangulation $T$ of $\Delta_{n-1}\times \Delta_{d-1}$ induces
an acyclic system of permutations of $[n]$ on $K_d=\operatorname{graph}(\Delta_{d-1})$ and an acyclic system of permutations of $S$ on $K_n=\operatorname{graph}(\Delta_{n-1})$. We call them \emph{dual} systems.
Both contain the same information about $T$, namely the way in which $T$ triangulates each square face
$ij\times IJ$. (That is, the restriction of $T$ to the $2$-skeleton).
One system can be retrieved from the other as follows: To retrieve the permutation of $S$ to be associated to the edge $ij$ from the system of permutations of $[n]$ on the edges of $K_d$, restrict the latter to the symbols $i$ and $j$. This induces an orientation of the complete graph $K_d$ (orienting every edge from $i$ to $j$) which is acyclic by assumption. Hence, it in turn induces a permutation of the vertices of $K_d$.
Figure~\ref{fig:permutations} shows two dual acyclic systems of permutations for $n=3$ and $d=4$.

\begin{figure}[ht]
	\centering
	\includegraphics[scale=0.7]
	{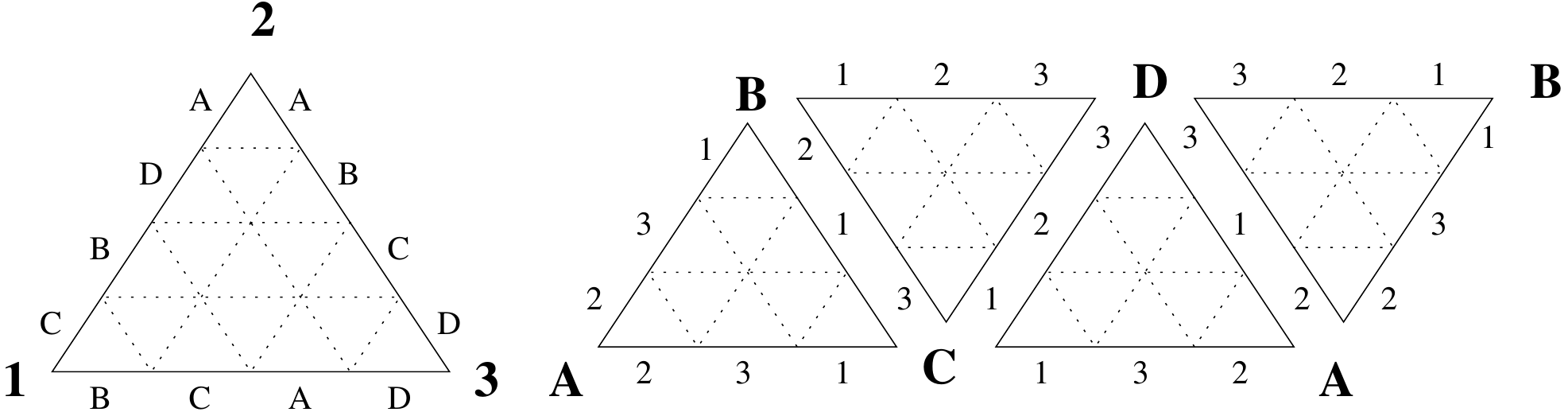}
	\bigskip\\	
	\caption{Two (dual) acyclic systems of permutations for the product $\Delta_2\times\Delta_3$. On the left, three permutations of $ABCD$ written along the edges of $K_{[3]}$. On the right, six permutations of $\{1,2,3\}$ along the edges of $K_{\{A,B,C,D\}}$.}
	\label{fig:permutations}
\end{figure}

Ardila and Ceballos~\cite{ArCe11} conjecture that every acyclic system of permutations for $\Delta_{n-1}\times\Delta_{l-1}$ can be extended to a triangulation of the polytope. This conjecture is equivalent to any of the following:
\begin{itemize}
\item Every triangulation of the $3$-skeleton of $\Delta_{n-1}\times\Delta_{d-1}$ can be extended to a triangulation of the polytope.
\item Every triangulation of $\operatorname{graph}(\Delta_{n-1})\times\Delta_{d-1} \cup\Delta_{n-1}\times  \operatorname{graph}(\Delta_{d-1})$ can be extended to a triangulation of the polytope.
\end{itemize}

The conjecture is trivial when $\min\{n-1,d-1\}=1$ and~\cite{ArCe11} contains a proof  for the case $\min\{n-1,d-1\}=2$.
In Section~\ref{sec:acyclic} we show that the conjecture fails for $\Delta_{4}\times\Delta_{3}$.

\subsection{The spread-out simplices  conjecture}
\label{sec:spread}

Ardila and Billey are interested in the \emph{positions} of the unmixed simplices in a fine mixed subdivision of $n\Delta_{d-1}$. To be more specific, let us take as standard simplex the convex hull of the standard basis. That is,
\[
\Delta_{d-1}:=\{(x_1,\dots,x_d)\in \R^d: \sum x_i=1,\quad x_i\ge 0\ \forall i\}.
\]
Then, each unmixed simplex can be written as $v+\Delta_{d-1}$ for a non-negative integer vector $v$ with sum of entries equal to $n-1$. Indeed, the (labeled) unmixed simplex is the Minkowski sum of $\Delta_{d-1}$ and $n-1$ (perhaps repeated) vertices of $\Delta_{d-1}$. The vector $v$ is the sum of those vertices. Ardila and Billey made the observation that the positions of unmixed simplices in a fine mixed subdivision are always \emph{spread-out} in the following sense:
%
%
%

\begin{definition}
Let $U:=\{v_1,\dots,v_n\}\subset\Z_{\ge 0}^d\cap \{\sum x_i=n-1\}$ be a set of $n$ integer non-negative vectors all with sum of coordinates equal to $n-1$. We say that $U$ is \emph{spread-out} if for any subset of $k$ of them we have
\[
\sum_{i=1}^d \min_j\{ (v_j)_i\} \le n-k.
\]
\end{definition}

Put differently, the unmixed simplices in a fine mixed subdivision are spread-out if no $k$ of them are contained in a subsimplex of size smaller than $k$.

\begin{theorem}[\protect{\cite[Proposition 8.2]{ArBi07}}]
The unmixed simplices in a fine mixed subdivision are spread-out.
\end{theorem}

They also made the following conjecture:

\begin{conjecture}[\protect{Spread-out simplices conjecture~\cite[Conjecture 7.1]{ArBi07}}]
\label{conj:spread}
If a set of $n$ vectors $U$ in $d$ coordinates is spread-out then there is a mixed subdivision of $n\Delta_{d-1}$ having those vectors as the positions of unmixed simplices.
\end{conjecture}

Ardila and Ceballos~\cite{ArCe11} show that the positions of the unmixed simplices can be derived
from the  acyclic system of permutations associated to a fine mixed subdivision $T$ of $n\Delta_{d-1}$ in the following fashion: the coordinate corresponding to vertex $v$ of $\Delta_{d-1}$ in the position vector $v_i$ of the $i$-th unmixed simplex is 
\[
\#\{j\in [n]\setminus \{i\} : \text{the source  of the acyclic graph corresponding to symbols $i$ and  $j$ is $v$}\}.
\]

Moreover, they show that any acyclic system of permutations (even the non-extendable one in Section~\ref{sec:acyclic}!) gives rise via that formula to a spread-out set of simplex positions.

\begin{example}[Example~\ref{ex:4D2} continued]
Let us see this rule in action in the mixed subdivision of Figure~\ref{fig:example}. To compute the position $v_1$ of the first triangle from the acyclic system of permutations, observe that the (dual) permutations of $ABC$ induced by $12$, $13$ and $14$ are, respectively, $CAB$, $ACB$, and $CAB$. The sources are two times $C$ and one time $A$, so the vector (written in the coordinates ordered as $ABC$) is $v_1=(1,0,2)$. Similarly, we compute
$v_2=(1,1,1)$,
$v_3=(0,2,1)$, and
$v_4=(0,3,0)$.
The latter, for example, corresponds to the fact that the triangle labeled $4$ is incident to vertex $B$, so that $4$ is the source in the three permutations induced by $41$, $42$ and $43$.
\end{example}

For $d=3$, the papers~\cite{ArCe11} and~\cite{ArBi07} show that the process is reversible: every spread-out system of positions of $n$ triangles in $n\Delta_2$ extends to a fine mixed subdivision and, in particular, to a system of permutations:

\begin{theorem}[\protect{Spread-out simplices conjecture~\cite[Theorem 6.2]{ArBi07},~\cite[Theorem 4.2]{ArCe11}}]
\label{thm:d=3}
If a set of $n$ vectors $U$ in $3$ coordinates is spread-out then there is a mixed subdivision of $n\Delta_{2}$ having those vectors as the positions of unmixed simplices.
\end{theorem}

The interest of Ardila and Billey in spread-out sets of simplices come from the following result of them:

\begin{theorem}[\protect{\cite{ArBi07}}]
Let $n$ and $d$ be two positive integers. Let $E_{n,d}:=\Z_{\ge 0}^d\cap \{\sum x_i=n-1\}$ be the set of possible positions for unmixed simplices in $n\Delta_{d-1}$
\begin{enumerate}
\item The subsets $\{U\in E_{n,d}: |U|=n, \textrm{ and $U$ is spread-out}\}$ are the bases of a matroid ${\mathcal T}_{n,d}$ of rank $n$ on $E_{n,d}$ (\cite[Theorem 4.1]{ArBi07}).
\item ${\mathcal T}_{n,d}$ is the matroid of lines in any sufficiently generic arrangement of $n$ flags in $\C^d$ (\cite[Theorem 5.1]{ArBi07}).
\end{enumerate}
\end{theorem}


\section{The Acyclic System Conjecture is false}

\subsection{A non-extendable boundary triangulation}
\label{sec:boundary}

The acyclic system conjecture would follow from the following statement: whenever $n+d>3$, every triangulation of the boundary of $\Delta_{n-1}\times\Delta_{d-1}$ extends to the interior. Here we show this statement is false, as a step towards disproving the conjecture.

Consider the acyclic system of permutations of Figure~\ref{fig:permutations}. Since $d-1=2$, the acyclic system conjecture is true in this case. In fact, the acyclic system of permutations extends to not one but three different triangulations, displayed in Figure~\ref{fig:3triangulations} in the form of mixed subdivisions of $4\Delta_2$. Observe that the three have their triangles in the same positions, as predicted by~\cite{ArCe11}: from the acyclic system of permutations of a triangulation of $\Delta_{n-1}\times\Delta_{d-1}$ the positions of the $n$ $(d-1)$-simplices of the corresponding mixed subdivision of $n\Delta_{d-1}$ can be deduced.

\begin{figure}[ht]
	\centering
	\includegraphics[scale=0.7]
	{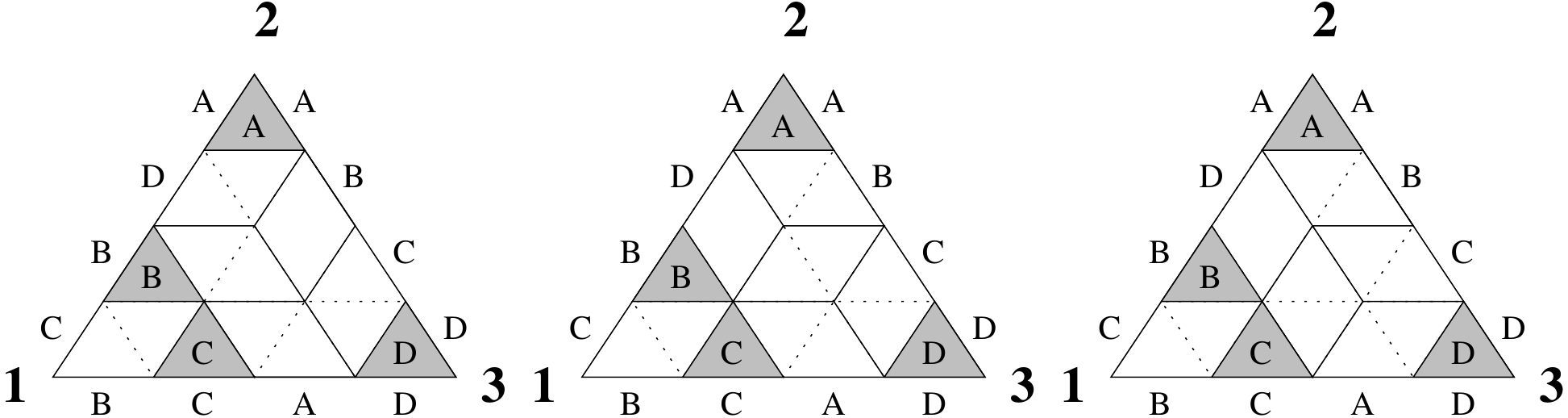}
	\bigskip\\	
	\caption{Three mixed subdivisions of $4\Delta_2$ (equivalently, three triangulations of $\Delta_3\times \Delta_2$) realizing the acyclic system of permutations of Figure~\ref{fig:permutations}.}
	\label{fig:3triangulations}
\end{figure}

But let us now consider how many triangulations of $\partial(\Delta_3\times \Delta_2)$ realize the same system of permutations. In $\Delta_3\times \Delta_2$ there are two types of facets: three copies of $\Delta_3\times \Delta_1$ and four copies of $\Delta_2\times \Delta_2$. In the former, the acyclic system of permutations already fixes the triangulations, since they are fixed by the dual system of permutations. But in the latter the only information that can be deduced from the system of permutations is what we see in Figure~\ref{fig:4triangulations}. In particular, there are four different ways of completing the acyclic system of permutations to a triangulation of $\partial(\Delta_3\times \Delta_2)$: each of the two hexagons of Figure~\ref{fig:4triangulations} can be tiled in two different ways.

\begin{figure}[ht]
	\centering
	\includegraphics[scale=0.7]
	{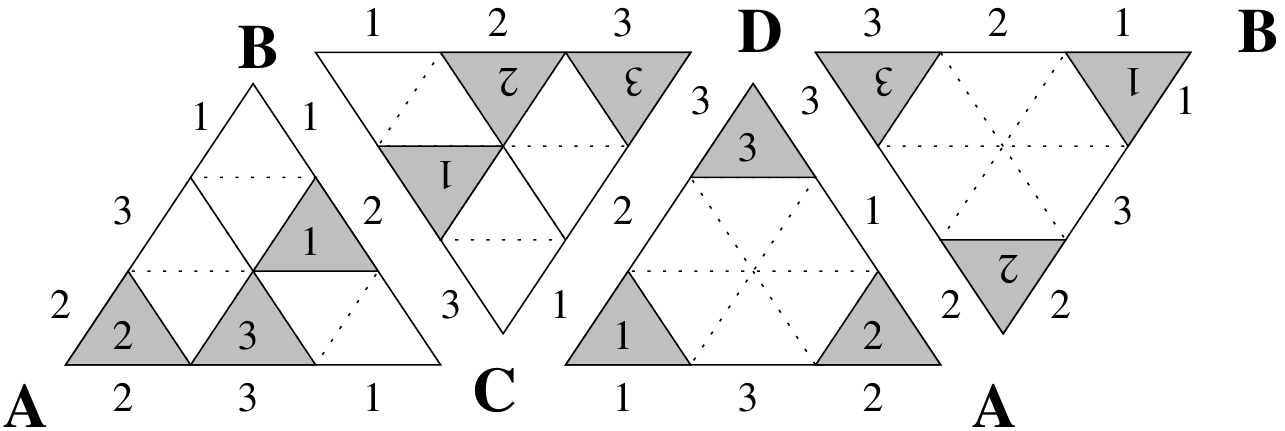}
	\bigskip\\	
	\caption{The boundary of $\Delta_3\times \Delta_2$ can be triangulated in four ways compatible with the acyclic system of permutations.
	}
	\label{fig:4triangulations}
\end{figure}

This simple counting implies that one of the four triangulations of $\partial(\Delta_3\times \Delta_2)$ cannot be extended to the interior, since only three triangulations of $\Delta_3\times \Delta_2$ realize the acyclic system of permutations. For future reference, let us explicitly show which one is not realizable, and why:

\begin{proposition}
The triangulation of the boundary of $\Delta_3\times \Delta_2$ displayed in Figure~\ref{fig:non-extendable} cannot be extended to a triangulation of $\Delta_3\times \Delta_2$.
\end{proposition}

\begin{figure}[ht]
	\centering
	\includegraphics[scale=0.7]
	{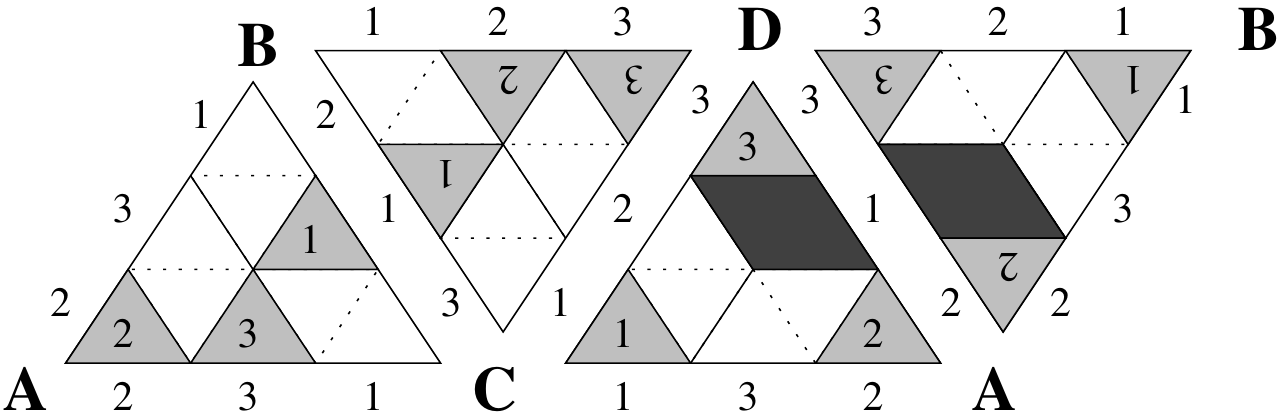}
	\bigskip\\	
	\caption{A triangulation of the boundary of $\Delta_3\times \Delta_2$.
	}
	\label{fig:non-extendable}
\end{figure}

\begin{proof}
Consider the two shaded rhombi of Figure~\ref{fig:non-extendable}. They correspond, respectively, to the simplices
$\{A1, D1, D2, A3, C3\}$ and $\{A1, D1, B2, D2, A3\}$ in $\Delta_3\times \Delta_2$. If this was extended to a triangulation $T$ of  $\Delta_3\times \Delta_2$, in $T$ these two simplices should be facets of a single simplex $\{A1, D1, B2, D2, A3, C3\}$ of $T$. In the mixed subdivision of $4\Delta_2$, this simplex would appear as the Minkowski sum $13 + 2 + 3 + 12$. That is, to
a rhombus with one side parallel to $12$ (and in the zone of $D$) and the other parallel to $13$ (and in the zone of $A$). Looking now at Figure~\ref{fig:3triangulations} we see that the three candidate triangulations have each a single rhombus with those properties.
But the simplices they represent are, respectively,
\[
\{A1, A3, B3, C3, D1, D2\}, \quad 
\{A1, A3, B3, C2, D1, D2\}, \quad \text{and} \quad
\{A1, A3, B2, C2, D1, D2\},
\]
instead of the one we need.
\end{proof}

\subsection{A non-extendable acyclic system of permutations}
\label{sec:acyclic}

We now use the previous example as a basis for a non-extendable acyclic system of permutations. The idea is to extend the system with two new symbols that force the boundary triangulation of Figure~\ref{fig:non-extendable} to arise. For this, consider the following acyclic system of permutations.

\begin{figure}[ht]
	\centering
	\includegraphics[scale=0.7]
	{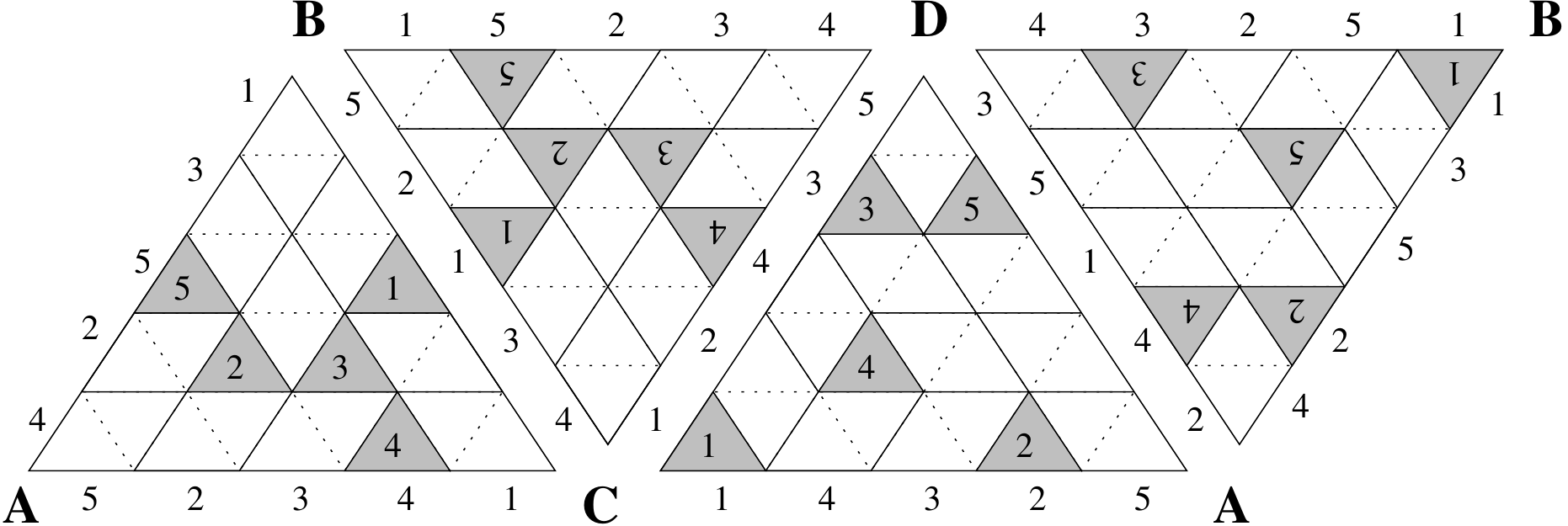}
	\bigskip\\	
	\caption{An acyclic system of permutations for the product $\Delta_4\times\Delta_2$ that cannot be extended to a triangulation.}
	\label{fig:boundary5}
\end{figure}

\begin{theorem}
The acyclic system of permutations of Figure~\ref{fig:boundary5} cannot be extended to a triangulation of $\Delta_4\times\Delta_2$.
\end{theorem}

\begin{proof}
In the figure we see not only the permutations but also an extension of them to four mixed subdivisions of $5\Delta_2$.
This (partial) extendability implies that the system is indeed acyclic. One important thing to notice is that in this particular example the extensions to $5\Delta_2$ are unique. To see this, remember that the positions of the five unit triangles in each mixed subdivision are unique by the general result of Ardila and Ceballos, and observe that in this examplee, once the triangles are positioned, the configuration of rhombi compatible with the permutations is also unique.

Suppose then that the acyclic system of permutations corresponds to a triangulation $T$ of $\Delta_4\times\Delta_3$. Then, $T$ restricts to the four facets of type $\Delta_4\times\Delta_2$ as shown in the figure. We now consider the \emph{deletion} of the symbols $4$ and $5$ in $T$. In the triangulation, this means that we restrict $T$ to the face $\Delta_3\times\Delta_2$ labeled by $\{1,2,3\}\times\{A,B,C,D\}$. In the mixed subdivisions, deletion corresponds to removing (or, collapsing to having zero width) the {zones} of the symbols $4$ and $5$. Doing so in Figure~\ref{fig:boundary5} gives precisely the non-extendable triangulation of Figure~\ref{fig:non-extendable}.
\end{proof}


\subsection{Some sufficient conditions for realizability of a spread-out system}
\label{sec:simplices}

The non-extendable acyclic system of permutations in Section~\ref{sec:acyclic} gives rise via that Ardila and Ceballos formula to a spread-out set of simplices. This would be a good candidate for a counter-example to the spread-out system conjecture.
Our first task is to show that it \emph{is not} a counter-example. This is based in the following realizability result.

Let $T_1$ be a fine mixed subdivision of $(n-1)\Delta_{d-1}$ and let $T_2$ be a fine mixed subdivision of $n\Delta_{d-2}$. Suppose that the restriction of $T_1$ to a certain facet $F$ of $\Delta_{d-1}$ coincides with the deletion of $n$ in $T_2$. Then there is a triangulation $T$ that extends both $T_1$ and $T_2$ (the former as a triangulation of $\Delta_{n-2}\times F$) to $\Delta_{n-1}\times \Delta_{d-1}$. Moreover, the positions of unmixed simplices in $T$ are as follows:
\begin{itemize}
\item For the last element $n$, $v_n$ is the same as it was in $T_2$, with a $0$ in the coordinate of the vertex opposite to $F$.
\item For the rest of elements, $v_i$ is the same as in $T_1$, with one unit added to the coordinate of the vertex opposite to $F$.
\end{itemize}

One can prove this directly in the world of mixed subdivisions of $n\Delta_{d-1}$, but a simpler proof can be done looking at them as triangultions of $\Delta_{n-1}\times \Delta_{d-1}$.
Let $v$ be the vertex opposite to $F$ in $\Delta_{d-1}$. Then, the only two facets of $\Delta_{n-1}\times \Delta_{d-1}$ not containing the vertex $(n,v)$ are $\Delta_{n-2}\times \Delta_{d-1}$ and $\Delta_{n-1}\times F$. Since $T_1$ and $T_2$ triangulate them and agree in the intersection, we can extend to a triangulation of $\Delta_{n-1}\times \Delta_{d-1}$ by just \emph{pulling} (i.e., coning) the triangulations $T_1$ and $T_2$ to the vertex $(n,v)$. We leave it to the reader to check that the effect on the positions of unmixed simplices is as we stated.

With this in mind we have the following:

\begin{lemma}[]
\label{lemma:simplices}
Let $U$ be spread-out set of $n$ nonnegative integer vectors, all with sum of coordinates $n-1$. Suppose that there is a coordinate  $i$ that is positive in all vectors of $U\setminus \{v_n\}$ (and then zero in $v_n$, or else the system would not be spread-out).

Then, $U$ is realizable by some triangulation if and only if the set $U'$ obtained deleting $v_n$ from $U$ and subracting one unit to coordinate $i$ of every other vector is realizable.
\qed
\end{lemma}

\begin{example}[The spread-out simplices of the counter-example to the Acyclic System Conjecture]
Let us compute the spread-out simplices of the acyclic system of permutations of Figure~\ref{fig:boundary5}. To compute
 $v_5$ observe that the permutations of $\{A,B,C,D\}$ induced respectively by $12$, $13$, $14$ and $15$ are
$BCDA$, $BCAD$, $BCDA$ and $CBAD$. The sources are three times $B$ and one time $C$, so the vector (written in the coordinates ordered as $ABCD$) is $v_1=(0,3,1,0)$. Similarly, we compute
$v_2=(2,1,1,0)$,
$v_3=(0,1,1,2)$,
$v_4=(1,0,2,1)$, and
$v_5=(1,2,0,1)$,

Since
the third coordinate is zero only on $v_5$, the lemma tells us that to realize $U$ it suffices to realize
\[
U'=\{
v_1=(0,3,0,0),
v_2=(2,1,0,0),
v_3=(0,1,0,2),
v_4=(1,0,1,1)\}.
\]
Now the second coordinate vanishes only in $v_4$, so to realize $U'$ it suffices to realize
\[
U''=\{
v_1=(0,2,0,0),
v_2=(2,0,0,0),
v_3=(0,0,0,2)
\}.
\]

But this is precisely the set of positions for the unmixed simplices in the triangulation(s) of Figure~\ref{fig:non-extendable}. To be more precise, in Figure~\ref{fig:non-extendable} we do not see the full triangulation, but we know that (three) triangulations with those positions of simplices exist from Figure~\ref{fig:3triangulations}. 
Hence, triangulations realizing the spread-out set $U$ exist.
\end{example}


Lemma~\ref{lemma:simplices} has the following interesting special case: Suppose that we have a system of positions $\{v_1,\dots, v_n\}$ in which a certain coordinate, say the $i$-th one, takes all its possible values (from $0$ to $n-1$). We call such a system $i$-spread.

\begin{corollary}
\label{coro:spread}
Every $i$-spread system of positions is spread-out, and it is realizable by some triangulation.
\qed
\end{corollary}

An $i$-spread system is a spread-out system in which the sum of the $i$-coordinates of the position vectors is as large as possible. Indeed, in a spread-out system there cannot be more than $k$ simplices with their $i$th coordinate greater or equal to $n-1-k$, and there are exactly $k$ for every $k$ if and only if the system is $i$-spread. It seems natural to look at the opposite case: the case when the $i$-th coordinate vanishes in every position vector. We call this the $i$-null case. Realizability in this case is easy to decide, by induction on $n$.

\begin{lemma}
\label{lemma:i-null}
Let $P=\{v_1,\dots,v_n\}$ be a system of positions for the simplices in $n\Delta{d-1}$. Assume it is $i$-null and 
let $P'$ be the system of positions in $n\Delta{d-2}$ obtained by deleting the $i$-th coordinate in every vector. Then:
\begin{enumerate}
\item $P'$ is spread-out if and only if $P$ is spread-out.
\item $P'$ is realizable by a fine mixed subdivision if and only if $P$ is.
\end{enumerate}
\end{lemma}

\begin{proof}
Part (1) is straightforward. For part (2), the ``if'' direction follows from restriction of a mixed subdivision realizing $P$ to the $i$-th facet of $n\Delta_{d-1}$. In general, the positions of the restricted mixed subdivision are not fixed by the positions of the big one, but in the $i$-null case all the unmixed simplices are incident to the $i$-th facet, so their positions are the same in both.

For the ``only if'' direction, we switch to the language of triangulations of $\Delta_{n-1}\times\Delta_{d-1}$. In this world, the unmixed simplices are the simplices incident to the faces of the form $\{v\}\times\Delta_{d-1}$. Being $i$-null means that all such simplices are
incident to the facet $\Delta_{n-1}\times F$, where $F$ is the $i$-th facet of $\Delta_{d-1}$. In particular, from any triangulation $T'$ of 
$\Delta_{n-1}\times \Delta_{d-2}$ we can construct one of $\Delta_{n-1}\times \Delta_{d-1}$ that is $i$-null as follows: Embed 
$\Delta_{n-1}\times \Delta_{d-2}$ as the facet $\Delta_{n-1}\times F$. Cone $T'$ to ant vertex of $\Delta_{n-1}\times \Delta_{d-1}$ not in that facet. Extend that to ta triangulation of $\Delta_{n-1}\times \Delta_{d-1}$. (The latter can be always done via, for example, the \emph{placing} procedure. See, e.~g.,~\cite{dLRaSa10}). In this construction the positions of the unmixed simplices of $T'$ are the restriction f those of the extended triangulation.
\end{proof}

Lemma~\ref{lemma:i-null} does not imply that every $i$-null spread-out system of positions is realizable by a fine mixed subdivision. But it does imply (together with Lemma~\ref{lemma:simplices}) the following. 

\begin{corollary}
If the spread-out simplices conjecture is false, a minimal counter-example to it must have:
\begin{itemize}
\item At least two positions incident to every facet of $n\Delta_{d-1}$ (that is, for each coordinate, at least two vectors with a zero on it).
\item At least one position not incident to every facet of $n\Delta_{d-1}$ (that is, for each coordinate, at one vector non-zero in it).
\end{itemize}
\end{corollary}

\end{document}